\documentclass[12pt,twoside]{amsart}
\usepackage{amsmath}
\usepackage{amsthm}
\usepackage{amsfonts}
\usepackage{amssymb}
\usepackage{latexsym}
\usepackage{mathrsfs}
\usepackage{amsmath}
\usepackage{amsthm}
\usepackage{amsfonts}
\usepackage{amssymb}
\usepackage{latexsym}
\usepackage{geometry}
\usepackage{dsfont}
\usepackage[dvips]{graphicx}
\usepackage{color}
\usepackage[all]{xy}

\date{}
\allowdisplaybreaks[4] \footskip=15pt
\renewcommand{\uppercasenonmath}[1]{}

\usepackage{graphicx,amssymb}
\usepackage[all]{xy}
\usepackage{amsmath}

\allowdisplaybreaks
\usepackage{amsthm}
\usepackage{color}

\theoremstyle{plain}
\newtheorem{theorem}{Theorem}[section]
\newtheorem{proposition}[theorem]{Proposition}
\newtheorem{lemma}[theorem]{Lemma}
\newtheorem{corollary}[theorem]{Corollary}
\theoremstyle{definition}
\newtheorem{example}[theorem]{Example}
\newtheorem{definition}[theorem]{Definition}
\newtheorem{question}[theorem]{Question}
\theoremstyle{definition}

\theoremstyle{remark}

\def\Mod{{\rm Mod}}

\newcommand{\Tor}{\mbox{\rm Tor}}

\newcommand{\prodi}{\prod_{i\in I}}

\newcommand{\Id}{\mathrm{Id}}

\def\m{\frak m}

\def\Hom{{\rm Hom}}
\def\Ext{{\rm Ext}}
\def\Tor{{\rm Tor}}

\def\Ker{{\rm Ker}}

\def\Im{{\rm Im}}
\def\Coker{{\rm Coker}}

\def\Id{{\rm Id}}

\begin{document}
\begin{center}
{\large  \bf Almost coherent rings}

\vspace{0.5cm}  \ Xiaolei Zhang

{\footnotesize School of Mathematics and Statistics, Tianshui Normal University, tianshui 741001, China
\\}
\end{center}

\bigskip
\centerline { \bf  Abstract}
\bigskip
\leftskip10truemm \rightskip10truemm \noindent

Inspired from the work of P. Scholze on the finiteness of \(\mathbf{F}_{p}\)-cohomology groups of proper rigid-analytic varieties over \(p\)-adic fields,  Zavyalov recently introduced the notion of  almost coherent rings, which plays a key role in the almost ring theory.
In this paper, we characterize almost coherent rings in terms of almost flat modules and almost absolutely pure modules, integrating numerous classical results into almost mathematics. Besides, we show that every almost coherent $R$-module is not almost isomorphic to a coherent $R$-module, giving a negative answer to a question proposed in \cite[B. Zavyalov,
{\it  Almost coherent modules and almost coherent sheaves},
Memoirs of the European Mathematical Society 19. Berlin: European Mathematical Society (EMS), 2025]{B25}.
\vbox to 0.3cm{}\\
{\it Key Words:} almost coherent ring;  almost flat module;   almost absolutely pure module; cover and envelope.\\
{\it 2020 Mathematics Subject Classification:} 13E99.

\leftskip0truemm \rightskip0truemm
\bigskip

\section{introduction}

Throughout this paper,  we fix a ``base'' commutative ring $R$ with an ideal $\m$ such that $\m^2= \m$. For a ring $R$,
we always do almost mathematics on $R$ with respect to $\m$.

To manufacture the essential machinery for his proofs of the major conjectures in $p$-adic Hodge Theory \cite{F88}, G. Faltings pioneered the theory of almost rings  in the late 1980s and 1990s.  These results describe the deep structure of the cohomology of algebraic varieties over $p$-adic fields.  The true testament to the power of almost ring theory came with the rise of perfectoid geometry, developed by P. Scholze \cite{S12}. The fundamental theorem of perfectoid geometry, the Tilting Correspondence, states that the geometry of a perfectoid algebra in characteristic zero is ``almost equivalent'' to the geometry of its tilt, a perfect algebra in characteristic $p$.

The idea of almost coherence is  inspired from the work of P. Scholze on the finiteness of \(\mathbf{F}_{p}\)-cohomology groups of proper rigid-analytic varieties over \(p\)-adic fields.  In \cite{S13}, P. Scholze obtained that there is an almost isomorphism
$$
\mathrm{H}^{i}\big(X,\mathbf{F}_{p}\big)\otimes\mathcal{O}_{C}/p\simeq^{a}\mathrm{H}^{i}\big(X,\mathcal{O}_{X_{a}}^{+}/p\big)
$$
for any proper rigid-analytic variety \(X\) over a \(p\)-adic algebraically closed field \(C\). This almost isomorphism allows us to reduce studying certain properties of \(\mathrm{H}^{i}\left(X,\mathbf{F}_{p}\right)\) for a \(p\)-adic proper rigid-analytic space \(X\) to studying almost properties of the cohomology groups \(\mathrm{H}^{i}\left(X,\mathcal{O}_{X_{a}}^{+}/p\right)\), or the full complex \(\mathbf{R}\Gamma(X,\mathcal{O}_{X_{a}}^{+}/p)\). In particular, Scholze shows that \(\mathrm{H}^{i}\left(X,\mathbf{F}_{p}\right)\) are finite groups by deducing it from almost coherence of \(\mathrm{H}^{i}\left(X,\mathcal{O}_{X_{a}}^{+}/p\right)\) over \(\mathcal{O}_{C}/p\).

Coherent rings, which are generalizations of Noetherian rings, are very basic and important in the theory of commutative algebras.  It is well-known that a ring $R$ is coherent if and only if any product of flat  $R$-modules (projective  $R$-modules, copies of $R$) is  flat, if and only if  every $R$-module has a flat preenvelope, if and only if any pure quotient of  absolutely pure $R$-modules is absolutely pure, if and only if any direct limit of  absolutely pure $R$-modules is absolutely pure, if and only if any $R$-module has an absolutely pure (pre)cover,  if and only if $\Hom_R(E,I)$ is  flat for any absolutely pure module $E$ and injective module $I$, if and only if $\Hom_R(\Hom_R(F,I_1),I_2)$ is flat  for any  flat module $F$     and any injective modules $I_1$ and $I_2$. Recently,  B. Zavyalov \cite{B25} formally introduced the concepts of almost coherent rings, modules and sheaves, which plays a key role in the almost ring theory.  The main motivation of this paper is to extend these classical results in coherent rings as above to almost coherent rings.  Besides, we show that every almost coherent $R$-module is not almost isomorphic to a coherent $R$-module, giving a negative answer to a question proposed in \cite{B25}.

\section{Preliminaries and a counterexample to a problem proposed by Zavyalov}

As our results concerns almost rings, we refer some basic notions from \cite{GR03,B25}.  An $R$-module $M$ is said to be \emph{almost zero}, if $\m M$ is the zero module.
The category $\Sigma_R$, which is the full subcategory of $\Mod_{R}$ of all $R$-modules consisting of all  almost zero $R$-modules, is a Serre subcategory of $\Mod_{R}$. So, one can introduce the quotient category, which is called the category of \emph{almost $R$-modules},
\[
\Mod^a_{R}:=\Mod_{R}/\Sigma_R.
\]
 We refer to elements of $\Mod^a_R$ as almost
$R$-modules or $R^a$-modules

Note that the localization functor
\[
(-)^a \colon \Mod_{R} \to \Mod^a_{R}
\]
is exact.
So we can call a sequence of $R$-modules $$\cdots\rightarrow M_1\rightarrow M_2\rightarrow \cdots\rightarrow M_n\rightarrow\cdots$$  almost exact provided that the corresponding  sequence of $R^a$-modules $$\cdots\rightarrow M^a_1\rightarrow M^a_2\rightarrow \cdots\rightarrow M^a_n\rightarrow\cdots$$
is exact in $\mathrm{Mod}^a_R$.
A morphism $f\colon M\to N$ is called an \emph{almost isomorphism} (resp.\ \emph{almost injection}, resp.\ \emph{almost surjection}) if the corresponding morphism $f^a\colon M^a\to N^a$ is an isomorphism (resp.\ an injection, resp.\ a surjection) in $\mathrm{Mod}^a_R$.
It follows by \cite[Lemma 2.1.8]{B25} that the morphism $f$ is an almost injection (resp.\ almost surjection, resp.\ almost isomorphism) if and only if $\operatorname{Ker}(f)$ (resp.\ $\operatorname{Coker}(f)$, resp.\ both $\operatorname{Ker}(f)$ and $\operatorname{Coker}(f)$) is an almost zero module.

\begin{lemma}\cite[Lemma 4.1]{Zuscoh-24}\label{sdir} Let $R$ be a ring and $s\in R$. Let $\{M_i\mid i\in\Gamma\}$ be a direct system of  $R$-modules. Then $$s\lim\limits_{\longrightarrow}M_i\cong \lim\limits_{\longrightarrow}sM_i.$$
\end{lemma}

\begin{proposition}\label{al-dirlim} The classes of almost zero modules, almost injections, almost surjections and almost isomorphisms are closed under direct limits. Moreover, the direct limits of almost  exact sequences is also almost  exact.
\end{proposition}
\begin{proof} Let $\{M_i\mid i\in\Gamma\}$ be a direct system of almost zero $R$-modules. Set $M=\lim\limits_{\longrightarrow}M_i$. Then for any $s\in \m$, $sM=s\lim\limits_{\longrightarrow}M_i\cong \lim\limits_{\longrightarrow}sM_i=0$ by Lemma \ref{sdir}. Let $\{f_i:M_i\rightarrow N_i\mid i\in\Gamma\}$ be a direct system of almost injection of $R$-modules. Set $f=\lim\limits_{\longrightarrow}f_i.$ Then for any $s\in\m$, we have $s\Ker(f)=s\Ker(\lim\limits_{\longrightarrow}f_i)\cong s\lim\limits_{\longrightarrow}\Ker(f_i)\cong \lim\limits_{\longrightarrow}s\Ker(f_i)=0$ by Lemma \ref{sdir} again.  Hence $f$ is an almost injection. The cases of almost surjections and almost isomorphisms are similar. The ``moreover'' part holds since direct limits is also commutative with kernels, images and cokernels.
\end{proof}
\begin{proposition}\label{inj-cog}
	Let $E$ be an injective cogenerator. Then the following statements are equivalent.
	\begin{enumerate}
		\item $T$ is  almost zero.
		\item   $\Hom_R(T,E)$  is almost zero.
	\end{enumerate}
\end{proposition}
\begin{proof} $(1)\Rightarrow (2)$: For any $s\in\m$ and any $f\in\Hom_R(T,E)$ and any $t\in T$, we have $sf(t)=f(st)=f(0)=0$. So $sf=0$ for any $s\in\m$, that is, $\Hom_R(T,E)$  is almost zero.

$(2)\Rightarrow (1)$: For any $s\in\m$, let $f_s:sT\rightarrow E$ be an $R$-homomorphism and $i:sT\rightarrow T$ be the embedding map. Since $E$ is injective, there exists an $R$-homomorphism  $g:T\rightarrow E$ such that $f_s=gi$. For any $t\in T$, we have $f_s(st)=sg(t)=0$ as $s\Hom_R(T,E)=0$. So  $\Hom_R(sT,E)=0$. Hence $sT=0$ since $E$ is an injective cogenerator. It follows that $T$ is  almost zero.
\end{proof}

The following is the ``five lemma'' in almost mathematics.
\begin{proposition}{\bf}\label{s-5-lemma}
Let $R$ be a ring. Consider the following commutative diagram with exact rows:
$$\xymatrix@R=20pt@C=20pt{
A\ar[d]_{f_A} \ar[r]^{g_1} & B\ar[d]_{f_B}\ar[r]^{g_2} &C\ar[d]^{f_C}\ar[r]^{g_3} &D\ar[r]^{g_4}\ar[d]^{f_D}&E\ar[d]^{f_E} \\
A'\ar[r]^{h_1} &B'\ar[r]^{h_2}  & C'  \ar[r]^{h_3} & D' \ar[r]^{h_4} & E'.\\}$$
\begin{enumerate}
\item  If $f_B$ and $f_D$ are almost monomorphisms and $f_A$ is an almost epimorphism, then $f_C$ is an almost monomorphism.
\item If $f_B$ and $f_D$ are almost epimorphisms and $f_E$ is an almost monomorphism, then $f_C$ is an almost epimorphism.
\item If $f_A$ is an almost epimorphism, $f_E$ is an almost monomorphism, and $f_B$ and $f_D$ are almost isomorphisms, then $f_C$ is an almost isomorphism.
\item  If $f_A$, $f_B$, $f_D$ and $f_E$ are all almost isomorphisms, then $f_C$ is an almost isomorphism.
\end{enumerate}
\end{proposition}
\begin{proof} The proof is similar with the classical case, and so we omit it.
\end{proof}

Recall from \cite[Definition 2.5.1]{B25} that an $R$-module $M$ is called almost finitely generated, if for any $s\in\m$ there is an integer $n_s$ and an $R$-homomorphism $$R^{n_s}\xrightarrow{f}M$$ such that $\Coker(f)$ is killed by $s.$  And recall from \cite[Definition 2.5.2]{B25} an $R$-module $M$ is called almost finitely presented, if for any $s,t\in\m$, there an integers $n_{s,t},m_{s,t}$ and a complex $$R^{m_{s,t}}\xrightarrow{g}R^{n_{s,t}}\xrightarrow{f}M$$
such that $\Coker(f)$ is killed by $s$ and $t\Ker(f)\subseteq \Im(g).$

\begin{definition}\cite[Definition 2.6.1, Definition 2.6.12]{B25} Let $R$ be a ring and $M$ an $R$-module.
\begin{enumerate}
\item An $R$-module $M$ is said to be almost coherent if is almost finitely generated and every almost almost finitely generated submodule $N^a\subseteq M^a$ is almost finitely presented.
\item A ring $R$ is called an almost coherent ring if $R$ itself is almost coherent as an $R$-module.
\end{enumerate}
\end{definition}

\begin{proposition}\label{char-ac}
A ring $R$ is almost coherent if and only if every finitely generated ideal of $R$ is almost finitely presented. Consequently, every almost coherent ring is coherent.
\end{proposition}
\begin{proof} We only need to show the sufficiency. Let $I$ be an almost finitely generated ideal of $R$. Then there is a finitely generated subideal $J$ of $I$ such that $I/J$ is almost zero. Denote by $i:J\hookrightarrow I$ the embedding map. Since $J$ is finitely generated, it is also  almost finitely presented by assumption. So for any $s,t\in\m$, there an integers $n_{s,t},m_{s,t}$ and a complex $$R^{m_{s,t}}\xrightarrow{g}R^{n_{s,t}}\xrightarrow{f}J$$
such that $\Coker(f)$ is killed by $s$ and $t\Ker(f)\subseteq \Im(g).$
It is easy to check that $$R^{m_{s,t}}\xrightarrow{g}R^{n_{s,t}}\xrightarrow{i\circ f}I$$
is the required  complex  for $I$,  that is, $I$ is almost finitely presented. Consequently, $R$ is an  almost coherent ring.
\end{proof}

Trivially, every coherent ring is almost coherent. However the converse is not true in general.
\begin{example} Let $R_1$ be a coherent ring and $R_2$ be a non-coherent ring. Set $R=R_1\times R_2$ and $\m=R_1\times 0.$ Then it is easy to verify that $R$ is an almost coherent ring but a non-coherent ring.
\end{example}

B. Zavyalov \cite{B25} shows that every $R$-module almost isomorphic to a coherent $R$-module is almost coherent via the following Lemma.
\begin{lemma}\cite[Lemma 2.6.5]{B25}
	Let $M$ be an $R$-module. If for any finitely generated ideal $ \mathfrak{m}_0 \subset \mathfrak{m} $, there exist a coherent $R$-module $N$  and a morphism $f : N \to M$ such that 
		$ \mathfrak{m}_0(\operatorname{Ker} f) = 0 $ and $ \mathfrak{m}_0(\operatorname{Coker} f) = 0$.
Then $M$ is an almost coherent module
\end{lemma}

Subsequently, he proposed the following Question: 
\begin{question}\cite[Question 2.6.6]{B25}
 Does the converse of above Lemma  hold?
\end{question}

In brief, he ask if every almost coherent $R$-module almost isomorphic to a coherent $R$-module. Now, we gave an negative answer to this question by the following counterexample.

\begin{example} Let $D=\mathbb{R}[x^{\frac{1}{n}}\mid n\geq 1]$ where $\mathbb{R}$ is the field of all real numbers. Then $Q:=\langle x^{\frac{1}{n}}\mid n\geq 1\rangle$ is a maximal ideal of $D$. Set $V=D_Q$.  Then $V$ is a valuation domain with maximal ideal $M:=Q_Q$ and valuation group $\mathbb{Q}$, the totally ordered group of all ration numbers. Set $R=\mathbb{Q}+M$ be a subring of $V$. It follows by \cite[Proposition 4.9]{D78-1} that $R$ is a pseudo-valutaion domain. Note  that $R$, whose  Krull dimension is $1$, is a non-coherent local domain with a non-finitely generated maximal ideal $M$ (see \cite[Corollary 8.37, Theorem 8.5.17]{fk16}). Set $\m:=M$, then $\m^2=\m$. It follows by \cite[Theorem 2.3]{D78} that $\m$ is a flat $R$-module and so is $\m\otimes_R \m.$
	
First, we claim that $R$ is an almost Noetherian ring, and so is almost coherent. Indeed, the only nonzero prime ideals of $R$ is $M$. Then for any $s\in\m$, we have $sM\subseteq \langle s\rangle \subseteq M$ and so $M$ is almost finitely generated. It follows by \cite[Corollary 3.3]{Zuano-25}  that  $R$ is an almost Noetherian ring, and thus is also an almost coherent ring  by \cite[Corollary 2.7.7]{B25}.

Next, we will show $R$ itself can not be  almost isomorphic to any coherent $R$-module. On contrary, assume that $R$ is almost isomorphic to a coherent $R$-module $P$. Then there there is an exact sequence of $R$-modules
$$0\rightarrow T_1\rightarrow P\xrightarrow{f} R\rightarrow T_2\rightarrow 0,$$
such that $sT_1=sT_2=0$ for any $s\in\m.$ Then $s(R/\Im(f))=sT_2=0$for any $s\in\m,$ and so $\m\subseteq \Im(f)$. Since $\m$ is maximal, we have $\Im(f)=\m$ or $\Im(f)=R$. Since $P$ is finitely generated, so is $\Im(f)$ and hence $\Im(f)=R$, that is, $T_2=0$. Since $R$ is projective, so the short exact sequence $0\rightarrow T_1\rightarrow P\rightarrow R\rightarrow 0$ is split. Since $P$ is a coherent $R$-module, so is $R$ itself implying $R$ is a coherent ring, which is a  contradiction.
 \end{example}

\section{Almost pure exact sequences}

\begin{definition}\label{def-s-f} Let $R$ be a ring.
	A short exact sequence $\Phi: 0\rightarrow A\rightarrow B\rightarrow C\rightarrow 0$ is said to be \emph{almost pure} provided that $\Phi^a: 0\rightarrow A^a\rightarrow B^a\rightarrow C^a\rightarrow 0$ is a pure exact sequence in $\mathrm{Mod}^a_R$, or equivalently, for any  $R$-module $M$, the induced sequence $$M\otimes_R\Phi: 0\rightarrow M\otimes_RA\rightarrow M\otimes_RB\rightarrow M\otimes_RC\rightarrow 0$$ is almost exact.
\end{definition}

Obviously, any  pure exact sequence is almost pure.  There are many characterizations of pure exact sequences in \cite[34.5]{w}. The next result generalizes some of those characterizations to almost pure exact sequences.

\begin{theorem}\label{c-s-pure}
	Let $\Phi:0\rightarrow A\xrightarrow{f} B\xrightarrow{f'} C\rightarrow 0$  be a  short exact sequence of $R$-modules.
	Then the following statements are equivalent.
	\begin{enumerate}
		\item $\Phi$ is almost pure.
\item For any finitely presented $R$-module $M$, the induced sequence $$M\otimes_R\Phi:0\rightarrow M\otimes_RA\rightarrow M\otimes_RB\rightarrow M\otimes_RC\rightarrow 0$$ is almost exact.
		\item If a system of equations $f(a_i)=\sum\limits^m_{j=1}r_{ij}x_j\ (i=1,\dots,n)$ with $r_{ij} \in R$ and unknowns $x_1,\dots, x_m$ has a solution in $B$, then   the system of equations  $sa_i=\sum\limits^m_{j=1}r_{ij}x_j\ (i=1,\dots,n)$ is solvable in $A$ for any $s\in\m$.
		\item  For any given commutative diagram with $F$ finitely generated free and $K$ a finitely generated submodule of $F$, for any $s\in\m$ there exists an  $R$-homomorphism $\eta_s:F\rightarrow A$  such that $s\alpha=\eta_s i$:
		$$\xymatrix@R=20pt@C=25pt{
			0\ar[r] &K\ar[d]_{\alpha}\ar[r]^{i}&F\ar@{.>}[ld]^{\eta_s}\ar[d]^{\beta}\\
			& A\ar[r]_{f} &B\\};$$
		\item For any finitely presented $R$-module $N$, the induced sequence $0\rightarrow\Hom_R(N,A)\rightarrow \Hom_R(N,B)\rightarrow \Hom_R(N,C)\rightarrow 0$ is  almost exact.
	\end{enumerate}
\end{theorem}
\begin{proof}  $(1)\Rightarrow(2)$: Trivial.

$(2)\Rightarrow(1)$: Let $M$ be an $R$-module. Then $M=\lim\limits_{\longrightarrow}M_i$, where $\{M_i\mid i\in\Gamma\}$ is a direct system of finitely presented $R$-modules. Then $0\rightarrow M_i\otimes_RA\xrightarrow{M_i\otimes_R f} M_i\otimes_RB\xrightarrow{M_i\otimes_R f'} M_i\otimes_RC\rightarrow 0$ is almost exact for each $i\in\Gamma$. And so $\Ker(M_i\otimes_R f)$ is almost zero. Consequently, $\Ker(M\otimes_R f)\cong\Ker(\lim\limits_{\longrightarrow}M_i\otimes_R f)\cong \lim\limits_{\longrightarrow}\Ker(M_i\otimes_R f)$ is almost zero by Lemma \ref{sdir} (or see Proposition \ref{al-dirlim} directly).

 $(2)\Rightarrow(3)$:  Let $M$ be a finitely presented $R$-module. Then there is an exact sequence $0\rightarrow K\rightarrow F\rightarrow M\rightarrow0$ with $F$ finitely generated free $R$-module and $K$ a finitely generated submodule of $F$. Then  $0\rightarrow M\otimes_RA\xrightarrow{1\otimes f} M\otimes_RB\rightarrow M\otimes_RC\rightarrow 0$ is almost exact by $(2)$.  So  $\Ker(1_M\otimes f)=0$, and hence $\Ker(1_{F/K}\otimes f)$ is almost zero. Now assume that there exists $b_j\in B$ such that $f(a_i)=\sum\limits^m_{j=1}r_{ij}b_j$ for any $j=1,\dots,m$. Assume that $\{e_1,\dots,e_n\}$ is the basis of  $F$ and suppose $K$ is generated by $m$ elements $\{\sum\limits^n_{i=1}r_{ij}e_i\mid j=1,\dots,m\}$. Then, $M=F/K$ is generated by $\{e_1+K,\dots,e_n+K\}$. Note that $\sum\limits^n_{i=1}r_{ij}(e_i+K)=\sum\limits^n_{i=1}r_{ij}e_i+K=0+K$ in $F/K$. Hence, we have $$\sum\limits^n_{i=1}((e_i+K)\otimes f(a_i))=\sum\limits^n_{i=1}((e_i+K)\otimes (\sum\limits^m_{j=1}r_{ij}b_j))=\sum\limits^m_{j=1}((\sum\limits^n_{i=1}r_{ij}(e_i+K))\otimes b_j)=0$$
	in $F/K\otimes B$. And so $\sum\limits^n_{i=1}((e_i+K)\otimes a_i)\in \Ker(1_{F/K}\otimes f)$.
	Hence, $s\sum\limits^n_{i=1}((e_i+K)\otimes a_i)=\sum\limits^n_{i=1}((e_i+K)\otimes sa_i)=0$ in $F/K\otimes_RA$ for any $s\in\m$. By \cite[Chapter I, Lemma 6.1]{FS01}, there exists $d^s_j\in A$ and $l^s_{ij}\in R$ such that $sa_i=\sum\limits^t_{k=1}l^s_{ik}d^s_k$ and $\sum\limits^n_{i=1}l^s_{ik}(e_i+K)=0$, and so  $\sum\limits^n_{i=1}l^s_{ik}e_i\in K$. Then there exists $t^s_{jk}\in R$ such that $\sum\limits^n_{i=1}l^s_{ik}e_i=\sum\limits^m_{j=1}t^s_{jk}(\sum\limits^n_{i=1}r_{ij}e_i)=\sum\limits^n_{i=1}(\sum\limits^m_{j=1}(t^s_{jk}r_{ij})e_i)$. Since $F$ is free, we have $l^s_{ik}=\sum\limits^m_{j=1}r_{ij}t^s_{jk}$. Hence
	$$sa_i=\sum\limits^t_{k=1}l^s_{ik}d^s_k=\sum\limits^t_{k=1}(\sum\limits^m_{j=1}r_{ij}t^s_{jk})d^s_k=
	\sum\limits^m_{j=1}r_{ij}(\sum\limits^t_{k=1}t^s_{jk}d^s_k)$$
	with $\sum\limits^t_{k=1}t^s_{jk}d^s_k\in A$. That is, $sa_i=\sum\limits^m_{j=1}r_{ij}x_j$ is solvable in $A$ for any $s\in\m$.

	$(3)\Rightarrow(2)$: Let $M$ be a finitely presented $R$-module. Then we have a exact sequence
	$M\otimes_RA\xrightarrow{1\otimes f} M\otimes_RB\rightarrow M\otimes_RC\rightarrow 0$. We will show that $\Ker(1\otimes f)$ is  almost zero. Let $\{\sum\limits^{n_\lambda}_{i=1}u^\lambda_i\otimes a^\lambda_i \mid \lambda\in \Lambda\}$ be the generators of $\Ker(1\otimes f)$. Then  $\sum\limits^{n_\lambda}_{i=1}u^\lambda_i\otimes f(a^\lambda_i)=0$ in $M\otimes_RB$ for each $\lambda\in \Lambda$. By \cite[Chapter I, Lemma 6.1]{FS01}, there exists $r^\lambda_{ij}\in R$ and $b^\lambda_j\in B$ such that $f(a^\lambda_i)=\sum\limits^{m_\lambda}_{j=1}r^\lambda_{ij}b^\lambda_j$ and $\sum\limits^{n_\lambda}_{i=1}u^\lambda_ir^\lambda_{ij}=0$ for each $\lambda\in \Lambda$. So $sa^\lambda_i=\sum\limits^{m_\lambda}_{j=1}r^\lambda_{ij}x^\lambda_j$  have a solution, say $a^{s,\lambda}_j$ in $A$ by (2) for any $s\in\m$. Then $$s(\sum\limits^{n_\lambda}_{i=1}u^\lambda_i\otimes a^{s,\lambda}_i)=\sum\limits^{n_\lambda}_{i=1}u^\lambda_i\otimes sa^{s,\lambda}_i=\sum\limits^{n_\lambda}_{i=1}u^\lambda_i\otimes (\sum\limits^{m_\lambda}_{j=1}r^\lambda_{ij}a^{s,\lambda}_j)=\sum\limits^{m_\lambda}_{j=1}((\sum\limits^{n_\lambda}_{i=1}r^\lambda_{ij}u^\lambda_i)\otimes a^{s,\lambda}_j)=0$$ for each $\lambda\in \Lambda$ and any $s\in \m$. Hence $s\Ker(1\otimes f)=0$ for any $s\in \m$. Consequently, $0\rightarrow M\otimes_RA\rightarrow M\otimes_RB\rightarrow M\otimes_RC\rightarrow 0$ is  almost exact.
	
	$(3)\Rightarrow(4)$:  Let  $\{e_1,\dots,e_n\}$ be the basis of the finitely generated free $R$-module $F$. Suppose $K$ is generated by $\{y_i=\sum\limits^{m}_{j=1}r_{ij}e_j\mid i=1,\dots,m\}$. Set $\beta(e_j)=b_j$ and $\alpha(y_i)=a_i$ for each $i$ and $j$. Then $f(a_i)=\sum\limits^{m}_{j=1}r_{ij}b_j$. By $(2)$, we have
	$sa_i=\sum\limits^m_{j=1}r_{ij}d^s_j$ for some $d^s_j\in A$. Let $\eta_s:F\rightarrow A$ be the $R$-homomorphism satisfying $\eta_s(e_j)=d^s_j$. Then $\eta_s i(y_i)=\eta i(\sum\limits^{m}_{j=1}r_{ij}e_j)=\sum\limits^{m}_{j=1}r_{ij}\eta(e_j)=\sum\limits^{m}_{j=1}r_{ij}d^s_j=sa_i=s\alpha(y_i)$,
	and so we have $s\alpha=\eta_s i$.
	
	$(4)\Rightarrow(5)$: Let $\gamma\in\Hom_R(N,C)$. Then there are $R$-homomorphisms $\delta:N\rightarrow B$ and $\alpha:K\rightarrow A$ such that the  following  diagram is commutative with rows exact:
	$$\xymatrix@R=20pt@C=30pt{
		0\ar[r] &K\ar[d]_{\alpha}\ar[r]^{i}&F\ar[r]^{\pi} \ar[d]^{\beta}& N\ar[d]^{\gamma}\ar[r]&0\\
		0\ar[r] & A\ar[r]_{f} &B\ar[r]_{f'}  &C\ar[r] & 0\\}$$
For any $s\in \m$. Then we have following  diagram is commutative with rows exact:
	$$\xymatrix@R=20pt@C=30pt{
		0\ar[r] &K\ar[d]_{s\alpha}\ar[r]^{i}&\ar@{.>}[ld]^{\eta_s}F\ar[r]^{\pi} \ar[d]^{s\beta}& N\ar@{.>}[ld]^{\delta_s}\ar[d]^{s\gamma}\ar[r]&0\\
		0\ar[r] & A\ar[r]_{f} &B\ar[r]_{f'}  &C\ar[r] & 0\\}$$
	By (4), there exists an homomorphism $\eta_s:F\rightarrow A$ such that $s\alpha=\eta_s i_K$. It follows by  \cite[Exercise 1.60]{fk16}, there is an $R$-homomorphism $\delta_s:N\rightarrow B$ such that $s\gamma=f'\delta_s$. Consequently, one can verify the  $R$-sequence $0\rightarrow \Hom_R(N,A)\rightarrow \Hom_R(N,B)\rightarrow \Hom_R(N,C)\rightarrow 0$ is almost exact.

	$(5)\Rightarrow(3)$:  Suppose that  $f(a_i)=\sum\limits^m_{j=1}r_{ij}b_j\ (i=1,\dots,n)$ with $a_i\in A$, $b_j\in B$ and  $r_{ij}\in R$. Let $F_0$ be a free module with a basis $\{e_1,\dots,e_m\}$ and  $F_1$ a free module with basis $\{e'_1,\dots,e'_n\}$. Then there are $R$-homomorphisms $\tau: F_0\rightarrow B$ and $\sigma: F_1\rightarrow \Im(f)$ satisfying $\tau(e_j)=b_j$ and $\sigma(e'_i)=f(a_i)$ for each $i,j$. Define an $R$-homomorphism $h:F_1\rightarrow F_0$ by $h(e'_i)=\sum\limits^m_{j=1}r_{ij}e_j$ for each $i$. Then $\tau h(e'_i)=\sum\limits^m_{j=1}r_{ij}\tau(e_j)=\sum\limits^m_{j=1}r_{ij}b_j=f(a_i)=\sigma(e'_i)$. Set $N=\Coker(h)$. Then $N$ is finitely presented. Thus there exists a homomorphism $\phi: N\rightarrow \Coker(f)$ such that the following diagram commutative:
	$$\xymatrix@R=20pt@C=30pt{
		&F_1\ar[d]_{\sigma}\ar[r]^{h}&F_0\ar[r]^{g} \ar[d]^{\tau}& N\ar@{.>}[d]^{\phi}\ar[r]&0\\
		0\ar[r] & A\ar[r]_{i} &B\ar[r]_{\pi}  &C\ar[r] & 0\\}$$
	Note that the induced sequence $$\Hom_R(N,B)\rightarrow \Hom_R(N,C)\rightarrow 0$$ is almost exact by (5). Hence for any $s\in\m$, there exists a homomorphism $\delta_s: N\rightarrow C$ such that $s\phi=\pi\delta_s$.
	Consider the following commutative diagram:
	$$\xymatrix@R=20pt@C=30pt{
		&F_1\ar[d]_{s\sigma}\ar[r]^{h}&F_0\ar@{.>}[ld]^{\eta_s}\ar[r]^{g} \ar[d]^{s\tau}& N\ar[ld]^{\delta_s}\ar[d]^{s\phi}\ar[r]&0\\
		0\ar[r] & A\ar[r]_{i} &B\ar[r]_{\pi}  &C\ar[r] & 0\\}$$
	We claim that there exists a homomorphism $\eta_s:F_0\rightarrow A$ such that $\eta_s f=s\sigma$. Indeed, since $\pi\delta_s g=s\phi g=\pi s \tau$, we have $\Im(s\tau-\delta_s g)\subseteq \Ker(\pi)=\Im(f)$. Define $\eta_s:F_0\rightarrow \Im(f)$ to be a homomorphism satisfying $\eta_s(e_i)=s\tau(e_i)-\delta_s g(e_i)$ for each $i$. So for each $e'_i\in F_1$, we have $\eta_s f(e'_i)=s\tau f(e'_i)-\delta_s g f(e'_i)=s\tau f(e'_i)$. Thus $i(s\sigma)=si\sigma=s\tau f=i\eta_s f$. Therefore, $\eta_s f=s\sigma$. Hence $sf(a_i)=s\sigma(e'_i)=\eta_s f(e'_i)=\eta (\sum\limits^m_{j=1}r_{ij}e_j)=\sum\limits^m_{j=1}r_{ij}\eta_s (e_j)$ with $\eta_s (e_j)\in \Im(f)$. So we have $sa_i=st'_1f(a_i) =\sum\limits^m_{j=1}r_{ij}t'_1\eta_s (e_j)$ with $t'_1\eta_s (e_j)\in A$ for each $i$.
\end{proof}

\section{Characterizing almost coherent rings in terms of almost flat modules}

In this section, we will investigate the notion of almost flat modules and characterize almost coherent rings in terms of almost flat modules.

\begin{definition}\cite[Definition 2.2.5]{B25}
  Let $R$ be a ring and $M$ an $R$-module.
\begin{enumerate}
\item  An $R^a$-module $M^a$ is flat if the functor $$M^a\otimes_{R^a}-:\Mod_R^a\rightarrow \Mod_R^a$$ is exact.
\item  An $R$-module $M$ is almost flat  if an $R^a$-module $M^a$
is flat as an  $R^a$-module.
\end{enumerate}
\end{definition}

\begin{proposition}\label{char-a-f} Let $R$ be a ring and $M$ an $R$-module. Then the following statements are equivalent.
\begin{enumerate}
\item $M$ is almost flat.
\item   $\Tor_1^R(N,M)$ is almost zero for every finitely presented $R$-module $N$.
\item   $\Tor_1^R(N,M)$ is almost zero for every $R$-module $N$.
\item    $\Tor_n^R(N,M)$ is almost zero every $R$-module $N$ and every $n\geq 1.$
\end{enumerate}
\end{proposition}
\begin{proof} $(1)\Rightarrow (2)$: Suppose $M$ is almost flat $R$-module. Consider the exact sequence $0\rightarrow K\rightarrow F\rightarrow M\rightarrow 0$ with $F$ a flat $R$-module. Then for any finitely presented $R$-module $N$, we have a long exact sequence
 $0\rightarrow\Tor_1^R(N,M)\rightarrow N\otimes_R K\rightarrow N\otimes_RF\rightarrow N\otimes_RM\rightarrow 0$. Since $M^a$ is flat, $\Tor_1^R(N,M)$ is almost zero. That is  $s\Tor_1^R(N,M)=0$  for every $s\in\m$ and every finitely presented $R$-module $N$.

$(4)\Rightarrow (3)\Rightarrow (2)$: Trivial.

$(2)\Rightarrow (1)$: Let $M$ be an $R$-module. For every almost exact sequence $0\rightarrow M_1\rightarrow M_2\rightarrow M\rightarrow 0$ of $R$-modules. Let $N$ be a finitely presented $R$-module. Then, by assumption, $0\rightarrow N\otimes_R K\rightarrow N\otimes_RF\rightarrow N\otimes_RM\rightarrow 0$ is almost exact, that is, the functor $M^a\otimes_{R^a}-:\Mod_R^a\rightarrow \Mod_R^a$ is exact.

$(2)\Rightarrow (3)$: Let $N$ be an $R$-module. Then $N=\lim\limits_{\longrightarrow}N_i$, where $\{N_i\mid i\in\Gamma\}$ is a direct system of finitely presented $R$-modules. By assumption, $\Tor_1^R(N_i,M)$ is almost zero, and so is $\Tor_1^R(N,M)$ by Lemma \ref{al-dirlim}.

$(3)\Rightarrow (4)$: Let $N$ be an $R$-module and $P_n\xrightarrow{f_n} P_{n-1}\xrightarrow{f_{n-1}} P_{n-2}\rightarrow\cdots\rightarrow P_0\xrightarrow{f_0} N\rightarrow 0$ be a projective resolution of $N$. Then, for every  every $n\geq 1,$ $\Tor_n^R(N,M)\cong \Tor_1^R(\Im(f_{n-1}),M)$, the latter is almost zero by assumption.
\end{proof}

\begin{corollary}\label{prop-af} Let $R$ be a ring. Then the following statements hold.
\begin{enumerate}
\item The class of almost flat modules is closed under almost isomorphisms.
\item The class of almost flat modules is closed under direct summands, direct sums, direct limits and extensions.
\end{enumerate}
\end{corollary}
\begin{proof}
	The proofs of these statements are directly, and so we omit it.
\end{proof}

\begin{lemma}\label{spq-f}
Any almost pure submodule and almost pure quotient module of an almost flat module are almost flat.
\end{lemma}
\begin{proof}
	Let $0\rightarrow A\rightarrow B\rightarrow C\rightarrow 0$ be an almost pure exact sequence with $B$ almost flat. Let $M$ be an arbitrary $R$-module. Then the result follows by the induced exact sequence

$\Tor_2^R(M,C)\rightarrow\Tor_1^R(M,A)\rightarrow\Tor_1^R(M,B)\rightarrow\Tor_1^R(M,C)\rightarrow M\otimes_RA\rightarrow M\otimes_RB\rightarrow M\otimes_RC\rightarrow0.$
\end{proof}

\begin{proposition} An $R$-module $F$ is almost flat if and only if any exact sequence $0\rightarrow A\rightarrow B\rightarrow F\rightarrow 0$ is  almost pure.
\end{proposition}
\begin{proof}  Let $M$ be a finitely presented $R$-module. Suppose $F$ is almost flat.  Then $\Tor_1^R(M,F)$ is almost zero. The result follows by the induced exact sequence $$\Tor_1^R(M,F)\rightarrow M\otimes_RA\rightarrow M\otimes_RB\rightarrow M\otimes_RF\rightarrow0.$$
	
On the other hand,	let $0\rightarrow A\rightarrow P\rightarrow F\rightarrow0$ be an exact sequence with $P$ a projective $R$-module. Then the result follows by the following exact sequence:  $$0\rightarrow\Tor_1^R(M,F)\rightarrow M\otimes_RA\rightarrow M\otimes_RP\rightarrow M\otimes_RF\rightarrow0.$$
\end{proof}

 Let $M$ be an $R$-module and $\mathscr{P}$ be a  class  of $R$-modules. Recall from \cite{gt} that  an $R$-homomorphism $f: M\rightarrow P$ with  $P\in \mathscr{P}$ is said to be a $ \mathscr{P}$-preenvelope  provided that  for any $P'\in \mathscr{P}$, the natural homomorphism  $\Hom_R(P,P')\rightarrow \Hom_R(M,P')$ is an epimorphism.
 
 It is well-known that a ring $R$ is coherent if and only if any product of flat $R$-modules is flat, if and only if any $R$-module has a flat preenvelope (see \cite[Theorem 3.2.24, Proposition 6.5.1]{EJ00}). Now, we characterized almost coherent rings in terms of almost flat modules.
 \begin{theorem}\label{newchase} Let $R$ be an ring. Then the following assertions are equivalent.
	\begin{enumerate}
		\item $R$ is an almost coherent ring.
		\item Any product of almost flat $R$-modules is almost flat.
		\item Any product of flat $R$-modules is almost flat.
		\item Any product of  projective $R$-modules is almost flat.
		\item Any product of copies of  $R$ is almost flat.
\item Any $R$-module has an almost flat preenvelope.
	\end{enumerate}
\end{theorem}
\begin{proof} $(1)\Rightarrow (2)$: Let $\{F_i\}$ be a family of almost flat $R$-modules. Let $M$ be a finitely presented $R$-module. Considering exact sequence $$0\rightarrow N\rightarrow P\rightarrow M\rightarrow 0$$ with $P$ finitely generated projective, we have $N$ is almost finitely generated by \cite[Lemma 2.5.15]{B25}. Since $R$ is almost coherent, it is easy to see $P$ is an almost coherent $R$-module. Hence $N$ is almost finitely presented. And so there exists
	an exact sequence of $R$-modules $$0 \rightarrow K\rightarrow Q\rightarrow N\rightarrow 0,$$ where $K$ is
	almost finitely  generated and $Q$ is  finitely generated projective by \cite[Lemma 2.5.15]{B25} again. Therefore, there exists
	an exact sequence of $R$-modules $$0 \rightarrow K'\rightarrow K\rightarrow T\rightarrow 0,$$ where $K'$ is finitely generated and $T$ is almost zero.
	
Consider the following commutative diagrams of exact sequences:
$$\xymatrix@R=20pt@C=25pt{
	0\ar[r]^{}&\Tor_1^R(M,\prod F_i)\ar[r]^{}\ar[d]^{\theta^1_M}&  N\otimes_R\prod F_i\ar[r]^{}\ar[d]^{\theta_N} & P\otimes_R\prod F_i\ar[d]^{\cong} &\\
	0\ar[r]^{}&\prod\Tor_1^R(M, F_i) \ar[r]^{}& \prod (N\otimes_RF_i) \ar[r]^{} &\prod (P\otimes_RF_i) & \\}$$
	
$$\xymatrix@R=20pt@C=25pt{
&K\otimes_R\prod F_i\ar[r]^{}\ar[d]^{\theta_K}&  Q\otimes_R\prod F_i\ar[r]^{}\ar[d]^{\cong} & N\otimes_R\prod F_i\ar[d]^{\theta_N}\ar[r]^{} &0\\
	&\prod (K\otimes_RF_i)\ar[r]^{}& \prod (Q\otimes_RF_i) \ar[r]^{} &\prod (N\otimes_RF_i)\ar[r]^{} &0 \\}$$	
and	
	$$\xymatrix@R=20pt@C=25pt{
	\Tor_1^R(T,\prod F_i)\ar[r]^{}\ar[d]^{\theta^1_T}	&K'\otimes_R\prod F_i\ar[r]^{}\ar[d]^{\theta_{K'}}&  K\otimes_R\prod F_i\ar[r]^{}\ar[d]^{\theta_K} & T\otimes_R\prod F_i\ar[d]^{\theta_T}\ar[r]^{} &0\\
	\prod\Tor_1^R(T, F_i)\ar[r]^{}&\prod (K'\otimes_RF_i)\ar[r]^{}& \prod (K\otimes_RF_i) \ar[r]^{} &\prod (T\otimes_RF_i)\ar[r]^{} &0.\\}$$
Since each $F_i$ is almost flat, $s\Tor_1^R(M, F_i)=0$, and hence $s\prod\Tor_1^R(M, F_i)=0$ for any $s\in\m,$ that is, $\prod\Tor_1^R(M, F_i)$ is almost zero. 	
Since $K'$ is finitely generated, $\theta_{K'}$ is an epimorphism. Since $T$ is almost zero, so are $ T\otimes_R \prod F_i$, $\prod (T\otimes_RF_i)$ ,	$\Tor_1^R(T,\prod F_i)$ and $	\prod\Tor_1^R(T, F_i)$. Hence,  $\theta_K$ is an almost epimorphism.  Then $\theta_N$ is an almost isomorphism. So  $\theta^1_M$
is also an almost isomorphism. Hence $\Tor_1^R(M,\prod F_i)$ is almost zero. Consequently, $\prod F_i$	is almost flat by Proposition \ref{char-a-f}.

$(2)\Rightarrow (3)\Rightarrow (4)\Rightarrow (5)$: Trivial.

$(5)\Rightarrow (1)$:   Let $J$ be a finitely generated ideal of $R$, $0\rightarrow J\rightarrow R\rightarrow R/J\rightarrow 0$ an exact sequence. Consider the following commutative diagram,

 $$\xymatrix{
&J\otimes_R \prod_{i\in I} R\ar[d]_{\phi_{J,\mathcal{R}}}\ar[r]^{f} & R\otimes_R  \prod_{i\in I}R \ar[d]_{\phi_{R,\mathcal{R}}}^{\cong}\ar[r]^{} & R/J\otimes_R  \prod_{i\in I} R \ar[d]_{\phi_{R/J,\mathcal{R}}}^{\cong}\ar[r]^{} &  0\\
   0 \ar[r]^{} & \prod_{i\in I} (J\otimes_RR)\ar[r]^{} &\prod_{i\in I} (R\otimes_RR ) \ar[r]^{} & \prod_{i\in I} (R/J\otimes_RR ) \ar[r]^{} &  0.\\}$$
Since $\prod_{i\in I} R$ is an almost flat module, then $f$ is an almost monomorphism. Thus $\Ker (f)=\Ker (\phi_{J,\mathcal{R}})$ is almost zero and then $\phi_{J,\mathcal{R}}$ is an almost monomorphism.

 Consider the exact sequence $0\rightarrow K\rightarrow F \rightarrow J\rightarrow 0$ with $F$ a finitely generated free module, then there is a commutative diagram of exact sequences,
  $$\xymatrix{
& K\otimes_R \prod_{i\in I} R\ar[d]_{\phi_{K,\mathcal{R}}}\ar[r]^{} & F\otimes_R  \prod_{i\in I} R \ar[d]_{\phi_{F,\mathcal{R}}}\ar[r]^{} & J\otimes_R  \prod_{i\in I} R \ar[d]_{\phi_{J,\mathcal{R}}}\ar[r]^{} &  0\\
   0 \ar[r]^{} & \prod_{i\in I} (K\otimes_R R )\ar[r]^{} &\prod_{i\in I} (F\otimes_R R ) \ar[r]^{} & \prod_{i\in I} (J\otimes_R R ) \ar[r]^{} &  0.\\}$$
Note that $\phi_{J,\mathcal{R}}$ is an almost monomorphism and $\phi_{F,\mathcal{R}}$ is an isomorphism, thus $\phi_{K,\mathcal{R}}$ is an almost epimorphism by Lemma \ref{s-5-lemma}.

We consider the exact sequence
$$\xymatrix{
K\otimes_R  R^K \ar[rr]^{\phi_{K,\mathcal{R}}}\ar@{->>}[rd] &&K^K \ar[r]^{} & T\ar[r]^{} &  0\\
    &\Im \phi_{K} \ar@{^{(}->}[ru] &&  &   \\}$$
with $T$ almost zero. Let $x=(k)_{k\in K}\in K^K$. Then $sx\in \Im \phi_{K,\mathcal{R}}$ for all $s\in\m$. Subsequently, for any $s\in\m$ and for any $i\in K$, there exists $k^s_j\in K, r^s_{j,i}\in R$ such that $sx=\phi_{K,\mathcal{R}}(\sum_{j=1}^n k^s_j\otimes (r^s_{j,i})_{i\in K})=(\sum_{j=1}^n k^s_j r^s_{j,i})_{i\in K}$. Set $L_s=\langle\{k^s_j|j=1,....,n \}\rangle$ be a finitely generated submodule of $K$. Now, for any $k\in K$, $sk=\sum_{j=1}^n k^s_j r^s_{j,k}\in L_s$, thus $sK\subseteq L_s\subseteq K$ for any $s\in\m$ and thus $K$ is almost finitely generated. And so $J$ is almost finitely presented by \cite[Lemma 2.5.15]{B25}. Consequently, $R$ is an almost coherent ring  by Proposition \ref{char-ac}.

$(6)\Rightarrow (2)$: Let 	$\{F_i\mid i\in\Lambda\}$ be a family of almost flat modules. Then, by assumption, $\prodi F_i$ has an almost flat preenvelope  $f:\prodi F_i\rightarrow F$ with $F$ almost  flat. Hence we have the following commutative diagram:
 $$\xymatrix{
 	 \prodi F_i\ar[r]^{f}\ar[d]^{\pi_i} & F \ar[ld]^{g_i} \\
  F_i&}$$
 Consequently, $\Id_{\prodi F_i}=(\prodi g_i)\circ f$. So  $ \prodi F_i$ is a direct summand of $F$.  It follows by Corollary \ref{prop-af} that $\prodi F_i$ is almost flat.

$(2)\Rightarrow (6)$: It follows by \cite[Corollary 6.2.2, Lemma 5.3.12]{EJ00}.	
\end{proof}

\section{almost absolutely pure modules}
Recall  that an $R$-module $M$ is said to be absolutely pure provided that $\Ext_R^1(N,M)=0$  for any finitely presented $R$-module $N$. Now, we introduce the notion of almost absolutely pure modules.

\begin{definition}\label{def-s-f} Let $R$ be a ring.
 An $R$-module $M$ is said to be almost absolutely pure provided that  $\Ext_R^1(N,M)$ is almost zero for any finitely presented $R$-module $N$.
\end{definition}

Next, we will give several characterizations of almost absolutely pure modules.
\begin{theorem}\label{c-s-abp}
	Let $R$ be a ring and $E$ an $R$-module.
	Then the following statements are equivalent.
	\begin{enumerate}
		\item $E$ is almost absolutely pure.
		\item  Any short exact sequence $0\rightarrow E\rightarrow B\rightarrow C\rightarrow 0$ beginning with $E$ is almost pure.
		\item  $E$ is an almost pure submodule in every injective module containing $E$.
		\item  $E$ is an almost pure submodule in its injective envelope.
		\item  If  $P$ is finitely generated projective, $K$ is a finitely generated submodule of $P$ and  $f:K\rightarrow E$ is an  $R$-homomorphism, then for some $s\in \m$ there is an $R$-homomorphism $g_s:P\rightarrow E$ such that $sf=g_si$  with $i:K\rightarrow P$  the natural embedding map.
	\end{enumerate}
\end{theorem}

\begin{proof} $ (2)\Rightarrow (3)\Rightarrow (4)$: Trivial.
	
	$(4)\Rightarrow(1)$: Let $I$ be the injective envelope of $E$. Then we have an almost pure exact sequence $0\rightarrow E\rightarrow I\rightarrow L\rightarrow 0$ by (4). Then, by  Theorem \ref{c-s-pure}, we have $0\rightarrow\Hom_R(N,E)\rightarrow \Hom_R(N,I)\rightarrow \Hom_R(N,L)\rightarrow 0$  is almost exact  for any finitely presented $R$-module $N$. Since  $0\rightarrow\Hom_R(N,E)\rightarrow \Hom_R(N,I)\rightarrow \Hom_R(N,L)\rightarrow \Ext_R^1(N,E)\rightarrow 0$ is exact. Hence $\Ext_R^1(N,E)$ is  almost zero for any finitely presented $R$-module $N$, that is, $E$ is almost absolutely pure.
	
	$(1)\Rightarrow(2)$:  Let  $N$ be a finitely presented $R$-module and $0\rightarrow E\rightarrow B\rightarrow C\rightarrow 0$ an exact sequence. Then there is an exact sequence $0\rightarrow\Hom_R(N,E)\rightarrow \Hom_R(N,B)\rightarrow \Hom_R(N,C)\rightarrow \Ext_R^1(N,E)$  for any finitely presented $R$-module $N$. By (1), $0\rightarrow\Hom_R(N,E)\rightarrow \Hom_R(N,B)\rightarrow \Hom_R(N,C)\rightarrow  0$  is almost exact  for any finitely presented $R$-module $N$.
	Hence $0\rightarrow E\rightarrow B\rightarrow C\rightarrow 0$ is almost pure by  Theorem \ref{c-s-pure}.
	
	$(1)\Rightarrow(5)$:   Considering the exact sequence $0\rightarrow K\xrightarrow{i} P\rightarrow P/K\rightarrow 0$, we have the following exact sequence $ \Hom_R(P,E)\xrightarrow{i_{\ast}} \Hom_R(K,E)\rightarrow \Ext_R^1(P/K,E)\rightarrow 0$. Since $P/K$ is finitely presented, $\Ext_R^1(P/K,E)$ is almost zero by $(1)$. Hence $i_{\ast}$ is an almost epimorphism, and so $s\Hom_R(K,E)\subseteq \Im(i_{\ast})$ for any $s\in\m$. Let $f:K\rightarrow E$ be an $R$-homomorphism. Then there is an $R$-homomorphism $g_s:P\rightarrow E$ such that $sf=g_si$.
	
	$(5)\Rightarrow(1)$:  Let $N$ be a finitely presented $R$-module.  Then we have an exact sequence $0\rightarrow K\xrightarrow{i} P\rightarrow N\rightarrow 0$ where $P$ is finitely generated projective and $K$ is finitely generated. Consider the following exact sequence $ \Hom_R(P,E)\xrightarrow{i_{\ast}} \Hom_R(K,E)\rightarrow \Ext_R^1(N,E)\rightarrow 0$.  By $(5)$, we have  $s\Hom_R(K,E)\subseteq \Im(i_{\ast})$. Hence $\Ext_R^1(N,E)$ is  almost zero for any finitely presented $R$-module $N$, that is, $E$ is almost absolutely pure.
\end{proof}

\begin{proposition}\label{flat-FP-injective}
	Let $R$ be a ring and $F$ an $R$-module. Then the following statements are equivalent.
	\begin{enumerate}
		\item $F$ is almost flat.
		\item   $\Hom_R(F,E)$ is almost absolutely pure for any injective module $E$.
		\item  If $E$ is an injective cogenerator, then $\Hom_R(F,E)$ is almost absolutely pure.
	\end{enumerate}
\end{proposition}
\begin{proof}

	$(1)\Rightarrow (2)$: Let $M$ be a finitely presented $R$-module and $E$ be an injective $R$-module. Since $F$ is almost flat,   $\Tor_1^R(M,F)$  is zero. Thus $\Ext_R^1(M,\Hom_R(F,E))\cong\Hom_R(\Tor_1^R(M,F),E)$ is also almost zero.  Thus $\Hom_R(F,E)$ is almost absolutely pure.
	
	$(2)\Rightarrow (3)$: Trivial.

	$(3)\Rightarrow (1)$:  Let  $E$ be an injective cogenerator. Since $\Hom_R(F,E)$ is almost absolutely pure, for any finitely presented $R$-module $M$ and  any $s\in \m$, we have  $$\Hom_R(\Tor_1^R(M,F),E)\cong \Ext_R^1(M,\Hom_R(F,E))$$ is almost zero. Since $E$ is an injective cogenerator, $\Tor_1^R(M,F)$ is almost zero  by Lemma \ref{inj-cog}.
\end{proof}
\begin{proposition}\label{spsub} Any direct products and direct sums of almost absolutely pure modules are almost absolutely pure.
\end{proposition}\label{dsdpaabs}
\begin{proof} Let $\{M_i\mid i\in\Gamma\}$ be a family of almost absolutely pure modules. Then for any $s\in \m$ and any finitely presented $R$-module $N$, we have
$$s\Ext_R^1(N,\bigoplus\limits_{i\in\Gamma}M_i)\cong s\bigoplus\limits_{i\in\Gamma}\Ext_R^1(N,M_i)\cong \bigoplus\limits_{i\in\Gamma}s\Ext_R^1(N,M_i)=0,$$
and
$$s\Ext_R^1(N,\prod\limits_{i\in\Gamma}M_i)\cong s\prod\limits_{i\in\Gamma}\Ext_R^1(N,M_i)\cong \prod\limits_{i\in\Gamma}s\Ext_R^1(N,M_i)=0.$$
\end{proof}
\begin{proposition}\label{spsub} Any almost pure submodule of an almost absolutely pure module is almost absolutely pure.
\end{proposition}
\begin{proof} Let $M_2$ be an almost absolutely pure module and $M_1$ be an almost pure submodule of $M_2.$ Let $K$ be a finitely generated submodule of a free module $F$. Let $\alpha:K\rightarrow M_1$ be an $R$-homomorphism. Consider the following commutative diagram
		$$\xymatrix@R=20pt@C=25pt{
			0\ar[r] &K\ar[d]_{s_1\alpha}\ar[r]^{i}&F\ar@{.>}[ld]^{\eta}\ar@{.>}[d]^{\beta_1}\\
			0\ar[r] & M_1\ar[r]_{f} &M_2\\}$$
Since $M_2$ is almost absolutely pure, for any $s_1\in \m$ there is  an $R$-homomorphism $\beta_1:F\rightarrow M_2 $ such that $s_1f\alpha=\beta_1 i$ by Theorem \ref{c-s-abp}. Since $f$ is almost pure,  for any $s_2\in \m$ there is  an $R$-homomorphism $\eta_2:F\rightarrow M_1 $ such that $s_1s_2\alpha=\eta_2 i$ by Theorem \ref{c-s-pure}. Since $\m=\m^2$, it is easily to check that for any $s\in \m$ there is  an $R$-homomorphism $\eta_s:F\rightarrow M_1 $ such that $s\alpha=\eta_s i$.
Hence $M_1$ is almost absolutely pure by Theorem \ref{c-s-abp}.
\end{proof}

\section{Characterzing almost coherent rings in terms of almost absolutely pure modules}

Let $M$ be an $R$-module and $\mathscr{E}$ be a  class  of $R$-modules. Recall from \cite{gt} that  an $R$-homomorphism $f: E\rightarrow M$ with  $E\in \mathscr{E}$ is said to be a $ \mathscr{E}$-precover  provided that  for any $E'\in \mathscr{E}$, the natural homomorphism  $\Hom_R(E',E)\rightarrow \Hom_R(E',M)$ is an epimorphism. Moreover any $R$-homomorphism $h:E\rightarrow E$ such that $f=f\circ h$ is an automorphism.

\begin{lemma}\label{s-split} Let $R$ be a ring, $s\in R$ and  $\xi: 0\rightarrow A\xrightarrow{f} B\xrightarrow{g} C\rightarrow 0$ a short exact sequence of $R$-modules. Then there is an $R$-homomorphism $f':B\rightarrow A$ such that  $f'\circ f=s\Id_A$  if and only if  there is  an $R$-homomorphism $g':C\rightarrow B$ such that  $g\circ g'=s\Id_C$.
\end{lemma}
\begin{proof} Suppose  there is an $R$-homomorphism $f':B\rightarrow A$ such that  $f'\circ f=s\Id_A$.  
	Define the map $g': C\rightarrow  B$ as follows.
	For any $z\in C$, pick $y\in B$ with $g(y) = z$ and define  $g'(z) = sy-f(f'(y))$.
	When  $g(y) = g(y') = z$, we pick $x\in A$ with $f(x) = y-y'$, so
	$(sy - f(f'(y))) - (sy' - f(f'(y'))) = s(y-y') - f(f'(y-y'))
	= sf(x) - f(f'(f(x))) = sf(x) - sf(x) = 0$, thus $g'$ is well-defined.
	It also can be checked that $g'$ is linear. Finally, if  $g(y) = z$,  we have
	$g(g'(z)) = g(sy - f(f'(y))) = sz$ because $g\circ f=0$. Then we have $g\circ g' = s\Id_C$. The sufficiency can be proved similarly.
\end{proof}

It is well-known that a ring $R$ is coherent if and only if any pure quotient of  absolutely pure $R$-modules is absolutely pure, if and only if any direct limit of  absolutely pure $R$-modules is absolutely pure, if and only if any $R$-module has an absolutely pure (pre)cover. Now, we give the similar result in almost mathematics.

\begin{theorem}\label{s-d-d-non}
Let $R$ be a ring. Then the following statements are equivalent.
\begin{enumerate}
    \item  $R$ is almost coherent.
    \item Any almost pure quotient of  almost absolutely pure $R$-modules is almost absolutely pure.
    \item Any pure quotient of  almost absolutely pure $R$-modules is almost absolutely pure.
      \item Any pure quotient of  absolutely pure $R$-modules is almost absolutely pure.
   \item  Any direct limit of  absolutely pure $R$-modules is almost absolutely pure.
   \item  Any direct limit of almost absolutely pure $R$-modules is almost absolutely pure.
   \item Any $R$-module has an almost absolutely pure precover.
   \item Any $R$-module has an almost absolutely pure cover.
\end{enumerate}
\end{theorem}
\begin{proof} $(1)\Rightarrow (2)$: Suppose  $R$ is an almost coherent ring. Let $B$ be an almost absolutely pure $R$-module. Let $C$ be an almost pure quotient of $B$ and $A$ be an  almost pure submodule of $B$. Let $P$ be a finitely generated projective $R$-module and $K$ finitely generated submodule of $P$. Let $i:K\rightarrow P$ be the natural embedding map,  and  $f_C:K\rightarrow C$ an $R$-homomorphism. We claim that for any $s_1\in\m$ there is an $R$-homomorphism $g_{1,C}:P\rightarrow C$ such that $s_1f_C=g_{1,C}i.$  Consider the following exact sequence
$0\rightarrow L\xrightarrow{i_L} P'\xrightarrow{\pi_{P'}} K\rightarrow 0$ with $P'$ finitely generated projective. Then we have the following exact sequence of commutative diagram:
$$\xymatrix@R=25pt@C=40pt{
	0\ar[r]^{}&L \ar[r]^{i_L}\ar[d]_{f_A} &P' \ar[d]^{f_B}\ar[r]^{} &\ar[d]^{f_C} K  \ar[r]^{}& 0 \\
	0\ar[r]&A\ar[r]^{i_A} &B\ar[r]^{\pi_B} & C \ar[r]^{}&0.\\ }$$
 Since $R$ is almost coherent, $L$ is almost finite, so there is an exact sequence $0\rightarrow K'\xrightarrow{i_{K'}} L\xrightarrow{\pi_L} T\rightarrow 0$ with $K'$ finitely generated and  $T$ is almost zero. Consider the following push-out:

$$\xymatrix@R=20pt@C=25pt{ & & 0\ar[d]&0\ar[d]&\\
	0 \ar[r]^{} & K'\ar@{=}[d]\ar[r]^{i_{K'}} & L\ar[d]_{i_L}\ar[r]&T\ar[d]\ar[r]^{} &  0\\
	0 \ar[r]^{} & K'\ar[r]^{} & P' \ar[d]\ar[r]^{} &X\ar[d]\ar[r]^{} &  0\\
&  & K \ar[d]\ar@{=}[r]^{} &K\ar[d] &  \\
	& & 0 &0 &\\}$$
Let $f_A:L\rightarrow A$ be an $R$-homomorphism. It follows by Proposition \ref{spsub} that $A$ is almost absolutely pure. Then for any $s_2\in \m$ there is an $R$-homomorphism $g_{2,A}:P'\rightarrow A$ such that $s_2f_Ai_{K'}=g_{2,A}i_Li_{K'}$. Hence $s_1s_2f_A=s_1g_{2,A}i_L$. It follows by  \cite[Exercise 1.60]{fk16} that we have the following commutative diagram:
$$\xymatrix@R=25pt@C=40pt{
	0\ar[r]^{}&L \ar[r]^{i_L}\ar[d]_{s_1s_2f_A} &P'\ar@{.>}[ld]^{s_1g_{2,A}} \ar[d]^{s_1s_2f_B}\ar[r]^{} &\ar[d]^{s_1s_2f_C} K \ar@{.>}[ld]^{g_{2,B}} \ar[r]^{}& 0 \\
	0\ar[r]&A\ar[r]^{i_A} &B\ar[r]^{\pi_B} & C \ar[r]^{}&0,\\ }$$
satisfying $s_1s_2f_C=\pi_Bg_{2,B}$.
Since $B$ is almost absolutely pure,  by Theorem \ref{c-s-abp} for any  $s_3\in \m$  there exists $R$-homomorphism $g'_{3,B}:P\rightarrow B$ such that the following diagram is commutative:
	$$\xymatrix@R=30pt@C=40pt{
	K\ar[r]^{i}\ar[d]_{s_3g_B}&P\ar@{.>}[ld]^{g'_{3,B}}\\
B&\\}$$.
Consequently, we have $s_1s_2s_3f_C=s_3\pi_Bg_{2,B}=\pi_Bg_{3,B}i=gi$, where $g=\pi_Bg_{3,B}$ for any $s_1,s_2,s_3\in\m$.  Since $\m=\m^2=\m^3$, it is easy to check that for any $s\in\m$ there is $g_s:P\rightarrow B$ such that $sf_C=g_si$.
It follows by Theorem \ref{c-s-abp} that $C$ is an almost absolutely pure $R$-module.

	$(2)\Rightarrow (3)\Rightarrow (4)$, $ (6)\Rightarrow (5)$ and $(8)\Rightarrow (7)$: Trivial.
		
	$(3)\Rightarrow (6)$:	Let $\{M_i\}_{i\in\Gamma}$ be a direct system of almost absolutely pure $R$-modules. Then there is an pure exact sequence  $0\rightarrow K\rightarrow \bigoplus\limits_{i\in\Gamma}M_i\rightarrow {\lim\limits_{\longrightarrow}}M_i\rightarrow 0$. Note that $\bigoplus\limits_{i\in\Gamma}M_i$ is almost absolutely pure, so is ${\lim\limits_{\longrightarrow}}M_i$ by (3).
	
	$ (4)\Rightarrow (5)$: Similar with $(3)\Rightarrow (6)$.

	$(5)\Rightarrow (1)$: Let $I$ be a finitely generated  ideal of $R$,  $\{M_i\}_{i\in\Gamma}$  a direct system of $R$-modules. Let $\alpha: I\rightarrow \lim\limits_{\longrightarrow }M_i$ be a homomorphism. For any $i\in \Gamma$, $E(M_i)$ is the injective envelope of $M_i$. Then $E(M_i)$ is absolutely pure. By (4), $\Ext_R^1(R/I, \lim\limits_{\longrightarrow }M_i)$ is almost zero. So for any $s\in\m$ there exists an $R$-homomorphism $\beta_s:R\rightarrow \lim\limits_{\longrightarrow }E(M_i)$ such that the following  diagram commutes:
	$$\xymatrix@R=20pt@C=25pt{
		0\ar[r]^{}&I\ar[r]^{}\ar[d]_{s\alpha}&  R\ar[r]^{}\ar@{.>}[d]^{\beta_s} & R/I\ar@{.>}[d]^{}\ar[r]^{} &0\\
		0\ar[r]^{}&{\lim\limits_{\longrightarrow }}M_i \ar[r]^{}&{\lim\limits_{\longrightarrow }}E(M_i)  \ar[r]^{} &{\lim\limits_{\longrightarrow }}E(M_i)/M_i\ar[r]^{} & 0.\\}$$
	Thus, by \cite[Lemma 2.13]{gt}, there exists $j\in \Gamma$, such that $\beta_s$ can factor through $R\xrightarrow{\beta_{j,s}} E(M_j)$. Consider the following commutative diagram:
	$$\xymatrix@R=20pt@C=25pt{
		0\ar[r]^{}&I\ar[r]^{}\ar@{.>}[d]_{s\alpha_j}&  R\ar[r]^{}\ar[d]^{\beta_{j,s}} & R/I\ar@{.>}[d]^{}\ar[r]^{} &0\\
		0\ar[r]^{}&M_j \ar[r]^{}&E(M_j)  \ar[r]^{} &E(M_j)/M_j\ar[r]^{} & 0.\\}$$
	Since the composition $I\rightarrow R\rightarrow E(M_j)\rightarrow E(M_j)/M_j$ becomes to be $0$ in the direct limit, we can assume $I\rightarrow R\rightarrow E(M_j)$ can factor through some  $I\xrightarrow{s\alpha_j} M_j$. Thus $s\alpha$ can factor through $M_j$. Consequently, the natural homomorphism  $\lim\limits_{\longrightarrow } \Hom_R(I,M_i)\xrightarrow{\phi}  \Hom_R(I, \lim\limits_{\longrightarrow }M_i)$ is an almost epimorphism. Now suppose  $\{M_i\}_{i\in\Gamma}$ is a direct system of finitely presented $R$-modules such that $\lim\limits_{\longrightarrow } M_i=I$. Then for any $s\in\m$ there exists $f_s\in \Hom_R(I,M_j)$ with $j\in \Gamma$ such that the identity map  $s\Id_I= \phi(u_j(f_s))$ where $u_j$ is the natural homomorphism $\Hom_R(I,M_j)\rightarrow \lim\limits_{\longrightarrow } \Hom_R(I,M_i)$. That is, for any $s\in\m$ we have the following commutative diagram: 	
	$$\xymatrix@R=10pt@C=40pt{
		I\ar[rd]^{f_s}\ar[dd]_{s\Id_I}&\\
		&M_i\ar[ld]^{\delta_i}\\
		I=\lim\limits_{\longrightarrow } M_i&\\}$$
We claim that $I$ is  almost finitely presented. Indeed, consider the short exact $0\rightarrow \Ker(\delta_i)\xrightarrow{i} M_i\xrightarrow{\pi} \Im(\delta_i)\rightarrow 0$.
	For any $s\in\m$, set $\pi'_s=f_s\circ i_I:\Im(\delta_i)\rightarrow M_i$ where $i_I:\Im(\delta_i)\rightarrow I$ is the embedding map. Then $\pi\circ \pi'_s=\pi\circ (f_s\circ i_I)=(\pi\circ f_s)\circ i_I=\pi\circ f_s|_{\Im(\delta_i)}=s\Id_{\Im(\delta_i)}.$ It follows by Lemma \ref{s-split}, there is an  $R$-homomorphism $i'_s:M_i\rightarrow\Ker(\delta_i)$ such that $i'_s\circ i=s\Id_{\Ker(\delta_i)}$ for any $s\in\m$. As $M_i$ is finitely  presented,  we have $\Ker(\delta_i)$ is almost finitely generated. Hence,  $\Im(\delta_i)$ is almost finitely presented, and so is $I$ by \cite[Corollary 2.5.12]{B25}. Consequently, $R$ is an almost coherent ring by Proposition \ref{char-ac}.

$(8)\Rightarrow (6)$: It follows by Proposition \ref{dsdpaabs}, Proposition \ref{spsub} and \cite[Theorem 3.4]{DD18}.

$(6)+(7)\Rightarrow (8)$: It follows
from \cite[Theorem 2.3.6]{EJ00}.

$(3)\Rightarrow (7)$:  Let $M$ be an $R$-module with $|M|=\lambda$. Then there is a cardinal $\kappa$ such that if $|M|\geq \kappa$ and $|M/L|\leq \lambda$, then $L$ contains a nonzero pure submodule.  We claim that for any $R$-homomorphism $f:A\rightarrow M$ with $A$ almost absolutely pure, there is an almost absolutely pure $R$-module $B$ with $|B|\leq \kappa$ such that, $f$ can factor through $B.$ Indeed, let  $f:A\rightarrow M$ be an $R$-homomorphism with $A $ almost absolutely pure. If $|A|<\kappa$, set  $B = A.$ Now suppose $|A|\geq\kappa$. Consider a submodule $S\leq A$ which is maximal such that that $S$ is pure in $A$ and $S \subseteq \Ker(f)$ hold. Let $B = A/S.$ Then there is an $R$-homomorphism $\overline{f}:B\rightarrow M$ such that $f=\overline{f}\circ \pi$, where $\pi:A\twoheadrightarrow B$ is the natural epimorphism.
It follows by (3) that $B$ is almost absolutely pure.  Next we will show $|B|\leq \kappa$. Indeed, let $K$ be
$\Ker(\overline{f})$. Then $|B/K|\leq |M|=\lambda.$ So if $|B|\geq \kappa$  there is a nonzero pure
submodule $T/S$ of $B/S$ contained in $K$. But then $T$ is pure in $A$ and is contained in
 $\Ker(f)$, contradicting the choice of $S$.

Now, let $E$ be the direct sums of all almost absolutely pure modules with cardinals at most $\kappa$ up to isomorphism. Then by the above, it is easy to verify $E$ is an almost absolutely pure precover of $M.$
\end{proof}

 It is well-known that  a ring $R$ is  coherent if and only if $\Hom_R(E,I)$ is  flat for any absolutely pure module $E$ and injective module $I$, if and only if $\Hom_R(\Hom_R(F,I_1),I_2)$ is flat  for any  flat module $F$     and any injective modules $I_1$ and $I_2$.

\begin{theorem}\label{phi-coh-fp} Let $R$ be a ring.  Then the following statements are equivalent.
	\begin{enumerate}
		\item $R$ is an  almost coherent ring.
		\item    $\Hom_R(E,I)$ is almost flat  for any    almost absolutely pure module $E$  and any injective module $I$.
		\item    If $I$  is an injective cogenerator, then   $\Hom_R(E,I)$ is almost flat  for any  almost absolutely pure module $E$.
		\item   $\Hom_R(\Hom_R(F,I_1),I_2)$ is almost flat  for any  almost flat module $F$     and any injective modules $I_1$ and $I_2$.
\item If $I_1$ and $I_2$ are injective cogenerators, then $\Hom_R(\Hom_R(F,I_1),I_2)$ is almost flat    for any   almost flat module $F$.
	\end{enumerate}
\end{theorem}
\begin{proof}
	$(2)\Rightarrow (3)$ and $(4)\Rightarrow (5)$: Trivial.
	
	$(2)\Rightarrow (4)$ and $(3)\Rightarrow (5)$: These follow by Proposition \ref{flat-FP-injective}.
	
	$(1)\Rightarrow (2)$: Suppose $R$ is an almost coherent ring. Let $M$ be a finitely presented $R$-module. Then there is an exact sequence $$0\rightarrow Q/K\rightarrow P\rightarrow M\rightarrow 0$$ with $P,Q$ finitely generated projective $R$-modules. And so the finitely presented $R$-module $P$ is an almost coherent by \cite[Lemma 2.6.14]{B25}. Hence the finitely generated submodule $Q/K$ of $P$ is almost finitely presented. So $K$ is almost finite.	And thus there exists	an exact sequence of $R$-modules $$0 \rightarrow T\rightarrow Q/K'\rightarrow Q/K\rightarrow 0,$$ where $K'$ is finitely generated and $T$ is  almost zero.
	
	 Consider the following commutative diagrams with exact rows ($(-,-)$ is instead of $\Hom_R(-,-)$):
	$$\xymatrix@R=20pt@C=15pt{
		(E,I)\otimes_R T\ar[r]^{} \ar[d]^{\psi^1_{T'}} &(E,I)\otimes_R Q/K' \ar[d]_{\psi_{Q/K'}}^{\cong}\ar[r]^{} & (E,I)\otimes_R Q/K \ar[d]_{\psi_{Q/K}}\ar[r]^{}&0\\
		((T,E),I)\ar[r]^{} & ((Q/K',E),I) \ar[r]^{} &((Q/K,E),I)\ar[r]^{} &0,\\ }$$
	and
	$$\xymatrix@R=20pt@C=15pt{
		0\ar[r]^{}&\Tor_1^R((E,I),M) \ar[r]^{}\ar[d]^{\psi^1_{M}} &(E,I)\otimes_R Q/K \ar[d]_{\psi_{Q/K}}\ar[r]^{} & (E,I)\otimes_R Q \ar[d]_{\psi_{Q}}^{\cong}\ar[r]^{} &(E,I)\otimes_R M \ar[d]_{\psi_{M}}^{\cong}\ar[r]^{} &0\\
		0\ar[r]&(\Ext_R^1(M,E),I) \ar[r]^{} &((Q/K,E),I)\ar[r]^{} & ((Q,E),I) \ar[r]^{} &((M,E),I)\ar[r]^{} &0\\ }$$
 It follows by chasing diagram and Proposition \ref{s-5-lemma} that $\psi^1_{M}$ is an almost isomorphism. Since $E$ is almost absolutely pure, $\Ext_R^1(M,E)$ is  almost zero, so is $\Hom_R(\Ext_R^1(M,E),I)$. Hence $\Tor_1^R(\Hom_R(M,E),M)$ is almost zero, and thus $\Hom_R(M,E)$ is almost flat by Proposition \ref{flat-FP-injective}.

$(5)\Rightarrow (1)$: Let $\{F_i\}$ be a family of flat $R$-module. We will show $\prod F_i$ is almost flat. Since $\oplus F_i$ is flat,  by assumption $$\Hom_R(\Hom_R(\oplus F_i,I_1),I_2)\cong \Hom_R(\prod\Hom_R(F_i,I_1),I_2)$$ is almost flat.  Note that $\oplus\Hom_R(F_i,I_1)$ is a pure submodule of $\prod\Hom_R(F_i,I_1)$. Then the natural epimorphism $$\Hom_R(\prod\Hom_R(F_i,I_1),I_2)\rightarrow \Hom_R(\oplus\Hom_R(F_i,I_1),I_2)\rightarrow 0$$ splits. Hence, $\prod\Hom_R(\Hom_R(F_i,I_1),I_2)\cong \Hom_R(\oplus\Hom_R(F_i,I_1),I_2)$ is also almost flat. Since $\prod F_i$ is a pure submodule of $\prod \Hom_R(\Hom_R(F_i,I_1),I_2)$, we have  $\prod F_i$ is almost flat by Lemma \ref{spq-f}. It follows by Theorem \ref{newchase} that $R$ is almost coherent.
\end{proof}

\end{document}